\documentclass{amsart}

\usepackage{esint}
\usepackage{amssymb}
\usepackage[abbrev,backrefs]{amsrefs}

\newtheorem{theorem}{Theorem}[section]
\newtheorem{lemma}[theorem]{Lemma}
\newtheorem{corollary}[theorem]{Corollary}

\theoremstyle{definition}
\newtheorem{definition}[theorem]{Definition}

\newtheorem{proposition}[theorem]{Proposition}

\theoremstyle{remark}

\numberwithin{equation}{section}

\begin{document}

\title{An Optimal Sobolev Embedding for $L^1$}


\author[D. Spector]{Daniel Spector}
\address{
Daniel Spector\hfill\break\indent
National Chiao Tung University\hfill\break\indent
Department of Applied Mathematics\hfill\break\indent
Hsinchu, Taiwan}
\address{
National Center for Theoretical Sciences\hfill\break\indent 
National Taiwan University\hfill\break\indent
No. 1 Sec. 4 Roosevelt Rd.\hfill\break\indent 
Taipei, 106, Taiwan}
\address{
Washington University in St. Louis  \hfill\break\indent 
Department of Mathematics and Statistics\hfill\break\indent
One Brookings Drive\hfill\break\indent 
St. Louis, MO}
\email{dspector@math.nctu.edu.tw}
\thanks{D.S. is supported in part by the Taiwan Ministry of Science and Technology under research grants 105-2115-M-009-004-MY2, 107-2918-I-009-003 and 107-2115-M-009-002-MY2.}


\subjclass[2010]{Primary }

\date{}

\dedicatory{}

\commby{}

\begin{abstract}
In this paper we establish an optimal Lorentz space estimate for the Riesz potential acting on curl-free vectors:  There is a constant $C=C(\alpha,d)>0$ such that
\[
\|I_\alpha F \|_{L^{d/(d-\alpha),1}(\mathbb{R}^d;\mathbb{R}^d)} \leq C \|F\|_{L^1(\mathbb{R}^d;\mathbb{R}^d)}
\]
for all fields $F \in L^1(\mathbb{R}^d;\mathbb{R}^d)$ such that $\operatorname*{curl} F=0$ in the sense of distributions.  This is the best possible estimate on this scale of spaces and completes the picture in the regime $p=1$ of the well-established results for $p>1$.
%
\end{abstract}

\maketitle
\section{Introduction}
The main result of this paper is
\begin{theorem}\label{mainresult}
Let $d\geq 2$ and $\alpha \in (0,d)$.  There exists a constant $C=C(\alpha,d)>0$ such that
\begin{align}\label{potentialnodiracl1}
\|I_\alpha F \|_{L^{d/(d-\alpha),1}(\mathbb{R}^d;\mathbb{R}^d)} \leq C \|F\|_{L^1(\mathbb{R}^d;\mathbb{R}^d)}
\end{align}
for all fields $F \in L^1(\mathbb{R}^d;\mathbb{R}^d)$ such that $\operatorname*{curl} F=0$ in the sense of distributions.
\end{theorem}
 Here $L^{d/(d-\alpha),1}(\mathbb{R}^d;\mathbb{R}^d)$ denotes the space of vector-valued functions whose Euclidean norm is in the Lorentz space $L^{d/(d-\alpha),1}(\mathbb{R}^d)$ (see below in Section \ref{preliminaries} for a precise definition of this space) and $I_\alpha$ is the Riesz potential, defined for measurable functions in the scalar setting by the formula
\begin{align*}
I_\alpha f(x) = \frac{1}{\gamma(\alpha)}\int_{\mathbb{R}^d} \frac{f(y)}{|x-y|^{d-\alpha}}\;dy,
\end{align*}
with an analogous definition in the vector setting by operating on components (see Section \ref{preliminaries} for the definition of the constant $\gamma(\alpha)$).  

As it may be of interest, let us also record two equivalent formulations of the inequality \eqref{potentialnodiracl1} before discussing the literature, our proof, some extensions, and a dual result.  In particular, taking into account the curl-free condition, the inequality \eqref{potentialnodiracl1} can alternatively be expressed as 
\begin{align}\label{gradientinequality}
\| I_\alpha \nabla u \|_{L^{d/(d-\alpha),1}(\mathbb{R}^d;\mathbb{R}^d)} \leq C\| \nabla u \|_{L^1(\mathbb{R}^d;\mathbb{R}^d)}
\end{align}
for all $u \in \dot{W}^{1,1}(\mathbb{R}^d)$ (which can be argued via Lemma 1 in \cite{BonamiPoornima}).  Such an estimate then extends to $\dot{BV}(\mathbb{R}^d)$ by density in the strict topology (and in turn one can also assert an analogue of \eqref{potentialnodiracl1} for measures).  Meanwhile the boundedness of the Riesz transforms on the Lorentz spaces imply that both \eqref{potentialnodiracl1} and \eqref{gradientinequality} are equivalent to  
\begin{align}\label{l1typeestimate}
\| I_\alpha f \|_{L^{d/(d-\alpha),1}(\mathbb{R}^d)} \leq C'\| R f \|_{L^1(\mathbb{R}^d;\mathbb{R}^d)}
\end{align}
for all distributions $f \in \mathcal{D}^\prime(\mathbb{R}^d)$ with $Rf:= \nabla I_1 f \in L^1(\mathbb{R}^d;\mathbb{R}^d)$.

Theorem \ref{mainresult} completes the picture concerning the study of the mapping properties of the Riesz potential on $L^p(\mathbb{R}^d)$ into Lorentz spaces for $1\leq p <\frac{d}{\alpha}$.  We recall that it was Sobolev who had initiated the study on the scale of Lebesgue spaces in \cite{sobolev}, where he demonstrated that one has the existence of a constant $\tilde{C}=\tilde{C}(\alpha,d)>0$ such that
\begin{align}\label{sobolev}
\|I_\alpha f \|_{L^q(\mathbb{R}^d)} \leq \tilde{C} \|f\|_{L^p(\mathbb{R}^d)}
\end{align}
for all $f \in L^p(\mathbb{R}^d)$, provided $1<p<d/\alpha$ and where
\[\frac{1}{q}=\frac{1}{p}-\frac{\alpha}{d}.\]
Subsequent work by O'Neil \cite{oneil} then showed that for the same range of $p$ and corresponding definition of $q$ one has an improvement to this inequality on the Lorentz scale, the inequality 
\begin{align}\label{oneil}
\|I_\alpha f \|_{L^{q,p}(\mathbb{R}^d)} \leq \tilde{C}' \|f\|_{L^p(\mathbb{R}^d)}
\end{align}
for some $\tilde{C}'>0$ and for all $f \in L^p(\mathbb{R}^d)$.  Recall that $L^q(\mathbb{R}^d) = L^{q,q}(\mathbb{R}^d)$, while spaces $L^{q,r}(\mathbb{R}^d)$ are nested increasing with respect to the second parameter.  The fact that $p<q$ thus implies that inequality \eqref{oneil} improves \eqref{sobolev}, while simple examples show that it is the best possible result on this scale.

It is well-known that \eqref{sobolev} (and hence \eqref{oneil}) cannot hold for $p=1$, though one has various possible replacements.  A classical result to this effect is the weak-type estimate of Zygmund \cite{zygmund}:  There exists $\tilde{C}''>0$ such that
\begin{align*}
|\{x: |I_\alpha f(x) | >t\}|^{(d-\alpha)/d} \leq \frac{\tilde{C}''}{t} \|f\|_{L^1(\mathbb{R}^d)}
\end{align*}
for all $t>0$ and all $f \in L^1(\mathbb{R}^d)$.  Here while the standard counterexample (cf. \cite{Sharmonic}, p.~119) shows that one cannot obtain a strong-type inequality with only the assumption $f \in L^1(\mathbb{R}^d)$, Stein and Weiss \cite{SteinWeiss} have shown that for $f$ in the Hardy space 
$\mathcal{H}^1(\mathbb{R}^d)$, one can obtain such a bound:  There exists $\tilde{C}'''>0$ such that
\begin{align*}
  \|I_\alpha f \|_{L^{d/(d-\alpha)}(\mathbb{R}^d;\mathbb{R}^d)} \leq \tilde{C}'''\int_{\mathbb{R}^d} \left|\left(f(x),Rf(x)\right)\right|\;dx
\end{align*}
for all $f \in \mathcal{H}^1(\mathbb{R}^d)$.  Observe here that we take as our definition of the Hardy space 
  \begin{align*}
 \mathcal{H}^1(\mathbb{R}^d):= \{ f \in L^1(\mathbb{R}^d) : Rf = \nabla I_1f \in L^1(\mathbb{R}^d;\mathbb{R}^d)\},
 \end{align*}
 though one has other possible definitions, for example, in terms of maximal functions \cite{FeffermanStein} or via an atomic decomposition \cite{Coifman,Latter}.  As Tartar has shown in \cite{tartar} that the Riesz potential maps atoms into the Lorentz space $L^{d/(d-\alpha),1}(\mathbb{R}^d)$, one can thus improve\footnote{Commenting on an earlier version of this manuscript, Mario Milman communicated to us a simple proof of this fact using the interpolation theory of Hardy spaces developed in \cite{FeffermanRiviereSagher}.} the preceding inequality to the optimal target on the Lorentz scale.
 
Yet while the assumption that both $f \in L^1(\mathbb{R}^d)$ and $Rf \in L^1(\mathbb{R}^d;\mathbb{R}^d)$ is sufficient to obtain a bound on the potential of $f$ in the suitably scaling Lebesgue space, it is not necessary, as been shown in recent work by the author, Armin Schikorra and Jean Van Schaftingen in \cite{SSVS}, where the following inequality was proven:  There exists a constant $C''=C''(\alpha,d)>0$ such that
\begin{align}\label{potentialnodirac}
\|I_\alpha f \|_{L^{d/(d-\alpha)}(\mathbb{R}^d)} \leq C'' \|Rf\|_{L^1(\mathbb{R}^d;\mathbb{R}^d)}
\end{align}
for all distributions $f \in \mathcal{D}^\prime(\mathbb{R}^d)$ such that $Rf\in L^1(\mathbb{R}^d;\mathbb{R}^d)$.  A comparison with the result of Tartar \cite{tartar} prompts one to wonder whether the inequality \eqref{potentialnodirac} can be strengthened on the Lorentz scale.  Indeed it can, as one sees from the formulation of Theorem \ref{mainresult} as the inequality \eqref{l1typeestimate} that one has precisely such an improvement.

As was remarked in \cite{SSVS}, one could already have deduced the inequality \eqref{potentialnodirac} from various embeddings in the literature which have been known for some time, e.g. \citelist{\cite{Solonnikov1975}*{Theorem 2}\cite{Kolyada}*{Theorem 
4}\cite{CDDD}*{Theorem 1.4}\cite{BourgainBrezisMironescu}*{Lemma 
D.2}\cite{mazya}\cite{VanSchaftingen2013}*{Theorem 8.3}}).  In fact, as was shown in \cite{SSVS}, one can even replace the norm of $I_\alpha f$ in $L^{d/(d-\alpha)}(\mathbb{R}^d)$ on the left-hand-side with its norm in $L^{d/(d-\alpha),r}(\mathbb{R}^d)$ for any $r>1$.  However, the constant in the theorem then depends upon $r$ and is not stable as $r\to 1^+$, and so one cannot obtain the optimal Lorentz space embedding with this argument.  Thus we can highlight the main achievements of Theorem \ref{mainresult}: to obtain the second parameter $r=1$ in the Lorentz space, to do so without the assumption $f \in \mathcal{H}^1(\mathbb{R}^d)$, and to accomplish these two feats for $\alpha \in (0,1)$.  Let us comment on these several facts here.  First, let us notice that to retain $r=1$ is significant, since only for $r=1$ does one have the embedding
\begin{align*}
I_\alpha : L^{d/\alpha,r}(\mathbb{R}^d) \to L^\infty(\mathbb{R}^d),
\end{align*}
(and even the space of continuous functions) as for any $r>1$ one obtains an embedding into the space of functions of bounded mean oscillation (and even a slightly better estimate involving local exponential integrability).  Second, the assumptions on $F$ in our Theorem \ref{mainresult} do not imply the underlying function $f=\operatorname*{div} I_1F  \in \mathcal{H}^1(\mathbb{R}^d)$.  A simple way to observe this fact is the lack of validity of the inequality
\begin{align*}
\|f\|_{L^1(\mathbb{R}^d)} \leq C \|Rf\|_{L^1(\mathbb{R}^d;\mathbb{R}^d)}.
\end{align*}
It is easy to construct a counterexample to such an inequality, for example, the sequence $Rf_n= \nabla u_n$, where $u_n=  \rho_n\ast \chi_{B(0,1)}$ for $\rho_n$ a sequence of standard mollifiers.  Then the right-hand-side remains bounded while
\begin{align*}
\int_{\mathbb{R}^d} |f_n(x)|\;dx = \int_{\mathbb{R}^d} |(-\Delta)^{1/2} u_n(x)|\;dx \to \infty.
\end{align*}
In particular, this construction exploits the fact that $(-\Delta)^{1/2}\chi_{B(0,1)}$ is a distribution whose (suitably defined) Riesz transform is a Radon measure, and not a function.  Finally regarding $\alpha \in (0,1)$:  Once one has established the validity of such an inequality for some $\alpha>0$, the result follows for all $\alpha'>\alpha$ from a vector-valued analogue of \eqref{oneil}.   As the case $\alpha=1$ can be deduced as a consequence of the result of Alvino \cite{Alvino}, the range $\alpha \geq 1$ follows from the existing literature.  In the sequel we therefore restrict our attention to the case $\alpha \in (0,1)$.


The idea of the proof is that while standard potential estimates are not sufficient to obtain an optimal exponent in the second parameter, the coarea formula allows for a sort of self-improvement through the estimate for characteristic funcitons.  The use of the coarea formula and isoperimetric inequalities in the proof of Sobolev inequalities in this spirit is classical \cite{FedererFleming,mazya1960}, while we here argue along the lines of a more recent work of Maz'ya \cite{mazya}.  To understand what is gained by such a reduction, let us suppose that we try to prove \eqref{gradientinequality} directly by our method, without assuming that one operates on characteristic functions.  

First, by a pointwise interpolation inequality of Maz'ya and Shaposhnikova \cite{mazyashaposhnikova} one has the following estimate:
For $\alpha \in (0,1)$, there exists a constant $C=C(\alpha,d)>0$ such that for each $u \in C^\infty(\mathbb{R}^d) \cap W^{1,1}(\mathbb{R}^d)$
\begin{align}\label{improvisedHedberg}
|I_\alpha \nabla u(x)| \leq C \left(\mathcal{M}(|\nabla u|)(x)\right)^{1-\alpha} \left(\mathcal{M}(u)(x)\right)^\alpha.
\end{align}
Next, by O'Neil's extension of H\"older's inequality in the Lorentz spaces \cite{oneil}, and moving to an equivalent quasi-norm in the Lorentz spaces (defined in terms of the distribution function, see below in Section \ref{preliminaries}), we can show one has the bound
\begin{align*}
\| I_\alpha \nabla u \|_{L^{d/(d-\alpha),1}(\mathbb{R}^d;\mathbb{R}^d)} \leq C' ||| \mathcal{M}(|\nabla u|)|||^{1-\alpha}_{L^{1,\infty}(\mathbb{R}^d)} |||\mathcal{M}(u) |||^\alpha_{L^{ d /(d-1),\alpha}(\mathbb{R}^d)}.
\end{align*}
Finally, by various weak and strong-type bounds of the Hardy-Littlewood maximal function on the Lorentz spaces one deduces
\begin{align}\label{almostoptimal}
\| I_\alpha \nabla u \|_{L^{d/(d-\alpha),1}(\mathbb{R}^d;\mathbb{R}^d)} \leq C''\| \nabla u \|^{1-\alpha}_{L^1(\mathbb{R}^d;\mathbb{R}^d)} |||u |||^\alpha_{L^{ d /(d-1),\alpha}(\mathbb{R}^d)}.
\end{align}
But as $\alpha<1$, the term $|||u |||^\alpha_{L^{ d /(d-1),\alpha}(\mathbb{R}^d)}$ is too large to be absorbed into  $\| \nabla u \|_{L^1(\mathbb{R}^d;\mathbb{R}^d)}$ for general $u$ (the case $\alpha=1$ is Alvino's result \cite{Alvino}).  

By passing to a limit in a suitable manner, however, we can obtain an analogue of \eqref{almostoptimal} for the characteristic function of a set of finite perimeter $E\subset \mathbb{R}^d$.  Here one finds that the independence of $|||\chi_E |||^\alpha_{L^{ d /(d-1),r}(\mathbb{R}^d)}$ with respect to $0<r\leq +\infty$ (up to a constant of equivalence that depends on $r$) allows one to regain the appropriate control of this term.  In fact, introducing the nonlinear fractional differential operator
\begin{align}\label{nonlineargradient}
\mathcal{D}^{1-\alpha}(u):=\int_{\mathbb{R}^d} \frac{|u(x)-u(y)|}{|x-y|^{d+1-\alpha}}\;dy,
\end{align}
defined for $u \in BV(\mathbb{R}^d)$, we can actually prove a stronger result (and easier to argue, due to positivity of the operator), the following 
\begin{lemma}\label{isoperimetricinequality}
Let $d\geq 2$ and $\alpha \in (0,1)$.  There exists a constant $C=C(\alpha,d)>0$ such that
\begin{align*}
\|\mathcal{D}^{1-\alpha}(\chi_E) \|_{L^{d/(d-\alpha),1}(\mathbb{R}^d)} \leq C Per(E)^{1-\alpha}|E|^{\alpha(1-1/d)}
\end{align*}
for all sets $E \subset \mathbb{R}^d$ of finite perimeter.
\end{lemma}

As discussed in \cite{SSVS}, Theorem \ref{mainresult} does not hold in the case $d=1$, and let us take this occasion to note where the assumption $d>1$ arises in the proof of Lemma \ref{isoperimetricinequality}.  It is in the step where we use H\"older's inequality in the Lorentz spaces, where the exponents are $p=1/(1-\alpha)$ and $q=d/\alpha(d-1)$:
\begin{align*}
\frac{1}{d/(d-\alpha)} = \frac{1}{1/(1-\alpha)} +\frac{1}{d/\alpha(d-1)}.
\end{align*}
In particular, in the case $d=1$ one has $q=+\infty$ and so one cannot pass to a weak-type estimate for the Hardy-Littlewood maximal function, instead requiring a strong-type estimate  on $L^1(\mathbb{R}^d)$, which is, of course, false.

Actually, by not invoking the isoperimetric inequality, our proof in Lemma \ref{isoperimetricinequality} obtains a more general result than the equivalence of isoperimetric and Sobolev inequalities discussed in \cite{mazya}.  In particular, it implies the general interpolation inequality given in our 
\begin{theorem}\label{interpolationinequality}
Let $\alpha \in (0,1)$.  There exists a constant $C=C(\alpha,d)>0$ such that
\begin{align*}
\|I_\alpha D u \|_{L^{d/(d-\alpha),1}(\mathbb{R}^d;\mathbb{R}^d)} \leq C \|D u\|^{1-\alpha}_{L^1(\mathbb{R}^d;\mathbb{R}^d)} \|u\|_{L^{d/(d-1),1}(\mathbb{R}^d)}^\alpha
\end{align*}
for all $u \in BV(\mathbb{R}^d)$.
\end{theorem}
Of course, one can then deduce further results by making other variations on this theme, possibly also employing known interpolation inequalities.  For example, as it answers a question raised in a previous work of the author and Tien-Tsan Shieh \cite{Shieh-Spector-2018}, we here extend the Hardy inequality proven there for $u \geq 0$ to $u$ with arbitrary sign in
\begin{theorem}\label{ineq:hardy}
Let $\alpha \in (0,1)$.   There exists a constant $C=C(\alpha,d)>0$ such that
\begin{align*}
 \int_{\mathbb{R}^d} \frac{|u|\;}{|x|^{\alpha}}\;dx \leq  C\int_{\mathbb{R}^d}|D^\alpha u|\;dx.
\end{align*}
for all $u$ such that $D^\alpha u = I_{1-\alpha} \nabla u  \in L^1(\mathbb{R}^d;\mathbb{R}^d)$. 
\end{theorem}
\noindent
The result in \cite{Shieh-Spector-2018} obtained the sharp constant for $u \geq 0$.  It would be interesting to understand whether one can show that the same constant for holds for unsigned $u$ (as in the case $\alpha=1$).

Let us make two further remarks here before moving to discuss dual results.  First, our proof obtains a slightly stronger result (see Theorem \ref{interpolationinequalityprime} in Section \ref{proof}): If $u \in W^{1,1}(\mathbb{R}^d)$ (or even $BV(\mathbb{R}^d)$) then in fact $\mathcal{D}^{1-\alpha}(u) \in L^{d/(d-\alpha),1}(\mathbb{R}^d)$.  One sees this is an improvement thanks to the easy inequality
\begin{align*}
\left|\int_{\mathbb{R}^d} \frac{u(x)-u(y)}{|x-y|^{d+1-\alpha}} \frac{x-y}{|x-y|}\;dy \right| \leq \int_{\mathbb{R}^d} \frac{|u(x)-u(y)|}{|x-y|^{d+1-\alpha}}\;dy,
\end{align*}
the left-hand-side being equal to $|I_\alpha Du|$, up to a multiplicative constant, in an appropriate sense.  Second, when one views Theorem \ref{mainresult} as the inequality \eqref{gradientinequality}, then an interesting fact (which could already be deduced from known embeddings) is made apparent:  While for $u \in L^1(\mathbb{R}^d)$ one has that
\begin{align*}
I_d u(x) := \frac{2}{\pi^{d/2} 2^d \Gamma(d/2)} \int_{\mathbb{R}^d} u(y) \log \frac{1}{|x-y|}\;dy
\end{align*}
is a function of bounded mean oscillation (see p.~417 in \cite{JohnNirenberg}), the assumption $\nabla u \in L^1(\mathbb{R}^d;\mathbb{R}^d)$ implies
\begin{align*}
I_d \nabla u(x) = \frac{2}{\pi^{d/2} 2^d \Gamma(d/2)} \int_{\mathbb{R}^d} \nabla u(y) \log \frac{1}{|x-y|}\;dy
\end{align*}
is a bounded function (that this holds for $d \in \mathbb{N}$ even has been commented by Van Schaftingen in \cite{VanSchaftingen2013}).

Finally we discuss a dual result concerning the mapping properties of the Riesz potentials which follows from Theorems \ref{mainresult} and \ref{interpolationinequality}.  In general, one has
\begin{align*}
I_\alpha : L^{d/\alpha,\infty}(\mathbb{R}^d) \to BMO(\mathbb{R}^d),
\end{align*}
for $BMO(\mathbb{R}^d)$ the space of functions of bounded mean oscillation.  Thus, the duality of the Hardy space $\mathcal{H}^1(\mathbb{R}^d)$ and $BMO(\mathbb{R}^d)$ implies that for any $g \in L^{d/\alpha,\infty}(\mathbb{R}^d)$, there exists functions $\{Y_j\}_{j=0}^d \subset L^\infty(\mathbb{R}^d)$ such that
\begin{align*}
I_\alpha g  = Y_0 + \sum_{j=1}^d R_j Y_j.
\end{align*}
For the canonical example of a reasonably smooth element of $L^{d/\alpha,\infty}(\mathbb{R}^d)$, the Riesz kernel $I_{d-\alpha}$, one has, in a suitable sense,
\begin{align*}
I_\alpha I_{d-\alpha}(x)  = \frac{2}{\pi^{d/2}2^d \Gamma(d/2)} \log|x| = \sum_{j=1}^d R_j Y_j
\end{align*}
for $Y_j = \frac{1}{(d-1)\gamma(d-1)} \frac{x_j}{|x|}$ (see, for example, \cite{GargSpector}).  One might suppose this is because of some benefit gained by the smoothness.  In fact, such a decomposition holds in general for elements in this space, that one does not need the $Y_0$:
\begin{corollary}\label{dual}
Let $\alpha \in (0,1)$.  There exists a constant $C=C(\alpha,d)>0$ such that for every $g \in L^{d/\alpha,\infty}(\mathbb{R}^d)$, there exists functions $\{Y_j\}_{j=1}^d \in L^\infty(\mathbb{R}^d)$ such that
\begin{align*}
I_\alpha g = \sum_{j=1}^d R_j Y_j
\end{align*}
with 
\begin{align*}
\|Y\|_{L^\infty(\mathbb{R}^d;\mathbb{R}^d)} \leq C\|g\|_{L^{d/\alpha,\infty}(\mathbb{R}^d)}.
\end{align*}
\end{corollary}
Results of this type have been pioneered by Bourgain and Brezis \cite{BourgainBrezis2002, BourgainBrezis2003, BourgainBrezis2004, BourgainBrezis2007}, and then subsequently studied by a number of authors (see, for example \cite{LanzaniStein2005}, \cite{BousquetMironescuRuss2013},  \cite{ChanilloVanSchaftingenYung}) in a far greater generality than we represent here.

The plan of the paper is as follows.  In Section \ref{preliminaries} we recall some background material on functions of bounded variation and on the Lorentz spaces.  For the former we recall some definitions, as well as the coarea formula.  For the latter we record useful versions of H\"older's and Young's inequalities one has on this scale.  In Section \ref{three} we give proofs of several lemmas that are useful in obtaining our result.  In Section \ref{proof} we prove Lemma \ref{isoperimetricinequality} and another intermediate result given in Theorem \ref{interpolationinequalityprime} before proceeding to prove Theorems \ref{mainresult}, \ref{interpolationinequality}, \ref{ineq:hardy}, and Corollary \ref{dual}.

\section{Preliminaries}\label{preliminaries}
In the Introduction we have defined the Riesz potential with a normalization constant $\gamma$.  We here recall that its value (see, e.g. \cite{Sharmonic}):
\begin{align*}
\gamma(\alpha):= \frac{\pi^{d/2}2^\alpha \Gamma\left(\frac{\alpha}{2}\right)}{\Gamma\left(\frac{d-\alpha}{2}\right)}.
\end{align*}

Let us now recall some results concerning the Lorentz spaces $L^{q,r}(\mathbb{R}^d)$.  We follow the convention of O'Neil in \cite{oneil}.  We being with some definitions related to the non-increasing rearrangement of a function.
\begin{definition}
For $f$ a measurable function on $\mathbb{R}^d$, we define
\begin{align*}
m(f,y):= |\{ |f|>y\}|.
\end{align*} 
As this is a non-increasing function of $y$, it admits a left-continuous inverse, called the non-negative rearrangment of $f$, and which we denote $f^*(x)$.  Further, for $x>0$ we define
\begin{align*}
f^{**}(x):= \frac{1}{x}\int_0^x f(t)\;dt.
\end{align*}
\end{definition}
With these basic results, we can now give a definition of the Lorentz spaces $L^{q,r}(\mathbb{R}^d)$.  In particular, we define
\begin{definition}
Let $1<q<+\infty$ and $1\leq r<+\infty$.  We define
\begin{align*}
\|f\|_{L^{q,r}(\mathbb{R}^d)} := \left( \int_0^\infty \left[t^{1/q} f^{**}(t)\right]^r\frac{dt}{t}\right)^{1/r},
\end{align*}
and for $1\leq q \leq+\infty$ and $r=+\infty$
\begin{align*}
\|f\|_{L^{q,\infty}(\mathbb{R}^d)} := \sup_{t>0} t^{1/q} f^{**}(t).
\end{align*}
\end{definition}
For these spaces, one has a duality between $L^{q,r}(\mathbb{R}^d)$ and $L^{q',r'}(\mathbb{R}^d)$ for $1<q<+\infty$ and $1\leq r < +\infty$.  This implies that one has
\begin{align*}
\| f\|_{L^{q,r}(\mathbb{R}^d)} = \sup \left\{ \left| \int_{\mathbb{R}^d} fg \;dx \right| : g \in L^{q',r'}(\mathbb{R}^d) \;\; \|g \|_{L^{q',r'}(\mathbb{R}^d)}\leq 1\right\},
\end{align*}
see, for example, Theorem 1.4.17 on p.~52 of \cite{grafakos}.

Let us observe that with this definition
\begin{align*}
\|f\|_{L^{1,\infty}(\mathbb{R}^d)} &= \|f\|_{L^1(\mathbb{R}^d)} \\
\|f\|_{L^{\infty,\infty}(\mathbb{R}^d)} &= \|f\|_{L^\infty(\mathbb{R}^d)},
\end{align*}
where the spaces $L^1(\mathbb{R}^d)$ and $L^\infty(\mathbb{R}^d)$ are intended in the usual sense.  It will be important for our purposes to have different endpoints than these, which is only possible through the introduction of a different object.  In particular, 
for $1<q<+\infty$, one has a quasi-norm on the Lorentz spaces $L^{q,r}(\mathbb{R}^d)$ that is equivalent to the norm we have defined.  What is more, this quasi-norm can be used to define the Lorentz spaces without such restrictions on $q$ and $r$.  Therefore let us introduce the following definition.
\begin{definition}
Let $1 \leq q <+\infty$ and $0<r<+\infty$ and we define
\begin{align*}
|||f|||_{\tilde{L}^{q,r}(\mathbb{R}^d)} :=  \left(\int_0^\infty \left(t^{1/q} f^*(t)\right)^{r} \frac{dt}{t}\right)^{1/r}.
\end{align*}
\end{definition}
Then one has the following result on the equivalence of the quasi-norm on $\tilde{L}^{q,r}(\mathbb{R}^d)$ and the norm on $L^{q,r}(\mathbb{R}^d)$ (and so in the sequel we drop the tilde):
\begin{proposition}
Let $1<q<+\infty$ and $1\leq r \leq +\infty$.  Then 
\begin{align*}
|||f|||_{\tilde{L}^{q,r}(\mathbb{R}^d)} \leq \|f\|_{L^{q,r}(\mathbb{R}^d)}\leq q' |||f|||_{\tilde{L}^{q,r}(\mathbb{R}^d)}.
\end{align*}
\end{proposition}
The proof is for $1 \leq r<+\infty$ can be seen by an application of Lemma 2.2 in \cite{oneil}, while the case $r=+\infty$ is an exercise in calculus (see also \cite{Hunt}, equation (2.2) on p.~258).

It will be useful for our purposes to observe an alternative formulation of this equivalent quasi-norm in terms of the distribution function.  In particular, Proposition 1.4.9 in \cite{grafakos} implies the following.
\begin{proposition}
Let $1<q<+\infty$ and $0<r<+\infty$.  Then
\begin{align*}
|||f|||_{L^{q,r}(\mathbb{R}^d)} \equiv q^{1/r} \left(\int_0^\infty \left(t |\{ |f|>t\}|^{1/q}\right)^{r} \frac{dt}{t}\right)^{1/r}.
\end{align*}
\end{proposition}

With either definition one can check the following scaling property that will be useful for our purposes (cf. Remark 1.4.7 in \cite{grafakos}):
\begin{align*}
|||\; |f|^\gamma |||_{L^{q,r}(\mathbb{R}^d)} = |||f|||^\gamma_{L^{\gamma q, \gamma r}(\mathbb{R}^d)}.
\end{align*}

With these definitions, we are now prepared to state H\"older's and Young's inequality on the Lorentz scale.  In particular on this scale one has a version of H\"older's inequality (Theorem 3.4 in \cite{oneil}):
\begin{theorem}\label{holder}
Let $f \in L^{q_1,r_1}(\mathbb{R}^d)$ and $g \in L^{q_2,r_2}(\mathbb{R}^d)$, where
\begin{align*}
\frac{1}{q_1}+\frac{1}{q_2}&=\frac{1}{q}<1\\
\frac{1}{r_1}+\frac{1}{r_2}&\geq  \frac{1}{r},
\end{align*}
for some $r \geq 1$.   Then
\begin{align*}
\|fg\|_{L^{q,r}(\mathbb{R}^d)} \leq q'\|f \|_{L^{q,r_1}(\mathbb{R}^d)}\|g \|_{L^{q',r_2}(\mathbb{R}^d)}
\end{align*}
\end{theorem}
We also have the following very useful generalization of Young's inequality (Theorem 3.1 in \cite{oneil}):
\begin{theorem}\label{young}
Let $f \in L^{q_1,r_1}(\mathbb{R}^d)$ and $g \in L^{q_2,r_2}(\mathbb{R}^d)$, and suppose $1< q<+\infty$ and $1\leq r\leq +\infty$ satisfy
\begin{align*}
\frac{1}{q_1}+\frac{1}{q_2}-1&=\frac{1}{q}\\
\frac{1}{r_1}+\frac{1}{r_2}&\geq \frac{1}{r}.
\end{align*}
Then
\begin{align*}
\|f\ast g\|_{L^{q,r}(\mathbb{R}^d)} \leq 3q \|f \|_{L^{q_1,r_1}(\mathbb{R}^d)}\|g \|_{L^{q_2,r_2}(\mathbb{R}^d)}.
\end{align*}
\end{theorem}

Here we utilize certain estimates for functions of bounded variation and sets of finite perimeter.  Let us here recall their definitions and some properties concerning them.  We define the space of functions of bounded variation as
\begin{align*}
BV(\mathbb{R}^d):= \left\{ u \in L^1(\mathbb{R}^d) :  \sup_{\Phi} \int_{\mathbb{R}^d} u \operatorname*{div} \Phi\;dx <+\infty \right\},
\end{align*}
where the supremum is taken over all
\begin{align*}
 \left\{\Phi \in C^1_c(\mathbb{R}^d;\mathbb{R}^d), \; \| \Phi \|_{L^\infty(\mathbb{R}^d;\mathbb{R}^d)} \leq 1\right\}.
\end{align*}
This definition implies the distributional derivative of $u$, which we denote by $Du$, is a Radon measure with finite total variation:
\begin{align*}
|Du|(\mathbb{R}^d)= \int_{\mathbb{R}^d} d|Du| <+\infty.
\end{align*}
We say that a set $E \subset \mathbb{R}^d$ has finite perimeter if $|E|<+\infty$ and $\chi_E \in BV(\mathbb{R}^d)$.  In particular, this implies that
\begin{align*}
Per(E):= |D\chi_E|(\mathbb{R}^d)= \sup \left\{ \left| \int_{\mathbb{R}^d} \chi_E \operatorname*{div} \Phi\;dx \right|  :\Phi \in C^1_c(\mathbb{R}^d;\mathbb{R}^d), \; \| \Phi \|_{L^\infty(\mathbb{R}^d;\mathbb{R}^d)} \leq 1\right\}<+\infty.
\end{align*}
For these functions, one has the product rule (see, for example, \cite{AFP}, p.~118, Proposition 3.2):
\begin{proposition}
Suppose $u \in BV(\mathbb{R}^d)$ and $\varphi \in C^1_c(\mathbb{R}^d)$.  Then
\begin{align*}
D(u\varphi) = Du \varphi + u \nabla \varphi \mathcal{L}^d.
\end{align*}
\end{proposition}
 One also has the coarea formula, whose proof can be found in \cite{AFP}, p.~144:  
 \begin{proposition}
 For $u \in BV(\mathbb{R}^d)$, the set $\{u>t\}$ has finite perimeter for almost every $t \in \mathbb{R}$ and
\begin{align*}
|Du|(\mathbb{R}^d) &= \int_{-\infty}^\infty |D\chi_{\{u>t\}}|(\mathbb{R}^d)\;dt \\
Du(\mathbb{R}^d) &= \int_{-\infty}^\infty D\chi_{\{u>t\}}(\mathbb{R}^d)\;dt.
\end{align*}
\end{proposition}

We also utilize some estimates and inequalities that involve the (centered) Hardy-Littlewood maximal function.  Here we recall its definition, which for a non-negative Radon measure $\mu$, is given by
\begin{align*}
\mathcal{M}(\mu)(x) := \sup_{r>0} \frac{1}{|B(x,r)|}\int_{\overline{B(x,r)}} d\mu.
\end{align*}
The Hardy-Littlewood maximal function enjoys several boundedness results that we emply here.  In particular, we require the standard weak-type estimate:
\begin{theorem}\label{weak-type}
There exists a constant $C=C(d)>0$ such that
\begin{align*}
\left|\left\{ x\in \mathbb{R}^d: \mathcal{M}(\mu)(x)>t\right\}\right| \leq \frac{C}{t} \int_{\mathbb{R}^d} \;d\mu
\end{align*}
for all $t>0$ and all non-negative Radon measures $\mu$.
\end{theorem}
The proof follows the standard one for functions in $L^1(\mathbb{R}^d)$, see for example \cite{Sharmonic}, p.~6.  In the introduction we asserted that one has the following bound for the Hardy-Littlewood maximal function in the Lorentz spaces (see Grafakos \cite{grafakos}, p.~56, Theorem 1.4.19):
\begin{theorem}
Let $1<q<+\infty$ and $0<r<+\infty$.  There exists a constant $C=C(r,q,d)>0$ such that
\begin{align*}
|||\mathcal{M}(f)|||_{L^{q,r}(\mathbb{R}^d)} \leq C |||f|||_{L^{q,r}(\mathbb{R}^d)} 
\end{align*}
for all $f \in L^{q,r}(\mathbb{R}^d)$.
\end{theorem}

\section{Several Lemmas}\label{three}
In this section we present the details of several estimates that we utilize in the proof of our main results.  The first is the following non-standard estimate for the Hardy-Littlewood maximal function, which is a variant of the bound on a Lorentz space $L^{q,r}(\mathbb{R}^d)$ for $1<q<+\infty$ and $r<1$. 
\begin{theorem}\label{strong-type}
Let $1<q<+\infty$ and $0<r<+\infty$.  There exists a constant $C=C(r,q,d)>0$ such that
\begin{align*}
|||\mathcal{M}(f)|||_{L^{q,r}(\mathbb{R}^d)} \leq C \|f\|_{L^\infty(\mathbb{R}^d)}^{1-1/q}  \|f\|_{L^1(\mathbb{R}^d)}^{1/q}
\end{align*}
for every $f \in L^1(\mathbb{R}^d) \cap L^\infty(\mathbb{R}^d)$.
\end{theorem}

\begin{proof}
From the definition we have
\begin{align*}
|||\mathcal{M}(f)|||_{L^{q,r}(\mathbb{R}^d)} = q^{1/r} \left(\int_0^\infty \left(t |\{ \mathcal{M}(f)>t\}|^{1/q}\right)^{r} \frac{dt}{t}\right)^{1/r}.
\end{align*}
As the Hardy-Littlewood maximal function satisfies the pointwise $L^\infty(\mathbb{R}^d)$ bound
\begin{align*}
\mathcal{M}(f) \leq \|f\|_{L^\infty(\mathbb{R}^d)}, 
\end{align*}
we find
\begin{align*}
q^{1/r} \left(\int_0^\infty \left(t |\{ \mathcal{M}(f)>t\}|^{1/q}\right)^{r} \frac{dt}{t}\right)^{1/r} = q^{1/r} \left(\int_0^{\|f\|_{L^\infty(\mathbb{R}^d)}} \left(t |\{ \mathcal{M}(f) >t\}|^{1/q}\right)^{r} \frac{dt}{t}\right)^{1/r}.
\end{align*}
Then as the standard weak-type estimate stated in Theorem \ref{weak-type} asserts
\begin{align*}
|\{ \mathcal{M}(f) >t\}| \leq \frac{C}{t} \int_{\mathbb{R}^d} |f|,
\end{align*}
we have
\begin{align*}
|||\mathcal{M}(f)|||_{L^{q,r}(\mathbb{R}^d)} &\leq q^{1/r} \left(\int_0^{\|f\|_{L^\infty(\mathbb{R}^d)}} \left(t \left(\frac{C}{t} \int_{\mathbb{R}^d} |f|\right)^{1/q}\right)^{r} \frac{dt}{t}\right)^{1/r} \\
&= 2Cq^{1/r}  \left(\int_0^{\|f\|_{L^\infty(\mathbb{R}^d)}} t^{r(1-1/q)-1} \left(\int_{\mathbb{R}^d} |f|\right)^{r/q}\;dt\right)^{1/r} \\
&=\frac{2Cq^{1/r}}{(r(1-1/q))^{1/r}} \|f\|_{L^\infty(\mathbb{R}^d)}^{1-1/q}\|f\|_{L^1(\mathbb{R}^d)}^{1/q}
\end{align*}
which completes the proof.
\end{proof} 

A key component of our argument is the following pointwise interpolation inequality for smooth functions, which in the $W^{1,1}(\mathbb{R}^d)$ case has been asserted in the paper of Maz'ya and Shaposhnikova \cite{mazyashaposhnikova}:
\begin{lemma}\label{mazyaslemma}
Let $\alpha \in (0,1)$.  There exists a constant $C=C(\alpha,d)>0$ such that \begin{align*}
\int_{\mathbb{R}^d} \frac{|u(x)-u(y)|}{|x-y|^{d+1-\alpha}}\;dy \leq C \left(\mathcal{M}(|\nabla u|)(x)\right)^{1-\alpha}  \left(\mathcal{M}(u)(x)\right)^{\alpha}
\end{align*}
for every smooth function $u \in W^{1,1}(\mathbb{R}^d)$.
\end{lemma}

 \begin{proof}
We split the integral into two pieces 
\begin{align*}
\int_{\mathbb{R}^d} \frac{|u(x)-u(y)|}{|x-y|^{d+1-\alpha}}\;dy&= \int_{B(x,r)} \frac{|u(x)-u(y)|}{|x-y|^{d+1-\alpha}}\;dy + \int_{B(x,r)^c} \frac{|u(x)-u(y)|}{|x-y|^{d+1-\alpha}}\;dy 
=: I+II.
\end{align*}
Now, for $I$ we let $\varphi \in C^\infty_c(B(x,2r))$ be a cutoff function such that $\varphi \equiv 1$ on $B(x,r)$ and $\|\nabla \varphi \|_{L^\infty(B(x,2r))} \leq \frac{C}{r}$.  Then by Hardy's inequality (\cite{Mazya:2011}, Equation 1.3.3) and the assumptions on the support of $\varphi$ we have
\begin{align*}
I &= \int_{B(x,r)} \frac{|(\varphi u)(x)-(\varphi u)(y)|}{|x-y|^{d+1-\alpha}}\;dy \\
&\leq \int_{\mathbb{R}^d} \frac{|(\varphi u)(x)-(\varphi u)(y)|}{|x-y|^{d+1-\alpha}}\;dy \\
&\leq C_1\int_{\mathbb{R}^d} \frac{|\nabla (\varphi u)(y)|}{|x-y|^{d-\alpha}}\;dy\\
&= C_1\int_{B(x,2r)} \frac{|\nabla (\varphi u)(y)|}{|x-y|^{d-\alpha}}\;dy.
\end{align*}
However, now the Leibniz rule, the $L^\infty(\mathbb{R}^d)$ bound on the derivative of $\varphi$, and the fact that $\nabla \varphi=0$ in $B(x,r)$ implies
\begin{align*}
I &\leq C_1\int_{B(x,2r)} \frac{|\nabla u(y)|}{|x-y|^{d-\alpha}}\;dy + C_1\int_{B(x,2r)} \frac{|\nabla \varphi (y)|  |u(y)|}{|x-y|^{d-\alpha}}\;dy \\
&\leq C_1\int_{B(x,2r)} \frac{|\nabla u(y)|}{|x-y|^{d-\alpha}}\;dy + \frac{C_1'}{r}\int_{B(x,2r) \setminus B(x,r)} \frac{ |u(y)|}{|x-y|^{d-\alpha}}\;dy \\
&=: III+IV.
\end{align*}
Concerning $III$, we apply the idea of Hedberg \cite{Hedberg} to make estimates on dyadic annuli:
\begin{align*}
III &\leq C_1 \sum_{i=-1}^\infty \int_{B(x,r/2^i)\setminus B(x,r/2^{i+1})} \frac{|\nabla u(y)|}{|x-y|^{d-\alpha}}\;dy \\
&\leq C_1 \sum_{i=-1}^\infty \left(r/2^{i+1}\right)^{\alpha-d} \int_{B(x,r/2^i)} |\nabla u(y)|\;dy\\
&= C_1 \sum_{i=-1}^\infty \left(r/2^{i+1}\right)^{\alpha-d} |B(0,1)| \left(r/2^{i}\right)^{d}  \fint_{B(x,r/2^i)} |\nabla u(y)|\;dy \\
&\leq C_1 \sum_{i=-1}^\infty \left(r/2^{i+1}\right)^{\alpha-d} |B(0,1)| \left(r/2^{i}\right)^{d} \mathcal{M}(|\nabla u|)(x).
\end{align*}
As one can sum the infinite series, we arrive at the estimate
\begin{align*}
III \leq C_2 r^\alpha \mathcal{M}(|\nabla u|)(x).
\end{align*}
For $IV$, we have
\begin{align*}
IV &= \frac{C_1'}{r}\int_{B(x,2r) \setminus B(x,r)} \frac{ |u(y)|}{|x-y|^{d-\alpha}}\;dy \\
&\leq \frac{C_1'}{r^{1-\alpha}} |B(0,1)| 2^d \fint_{B(x,2r)} |u(y)|\;dy \\
&\leq C_3 r^{\alpha-1} \mathcal{M}(u)(x),
\end{align*}
which shows
\begin{align*}
I \leq C_2 r^\alpha \mathcal{M}(|\nabla u|)(x)+C_3 r^{\alpha-1} \mathcal{M}(u)(x).
\end{align*}
Finally, we return to $II$ an apply the idea of Hedberg again, this time for large balls:
\begin{align*}
II &= \int_{B(x,r)^c} \frac{|u(x)-u(y)|}{|x-y|^{d+1-\alpha}}\;dy \\
&= \sum_{i=0}^\infty \int_{B(x,2^{i+1}r) \setminus B(x,2^i r)} \frac{|u(x)-u(y)|}{|x-y|^{d+1-\alpha}}\;dy \\
&\leq \sum_{i=0}^\infty (2^i r)^{-d-1+\alpha} \int_{B(x,2^{i+1}r) \setminus B(x,2^i r)} |u(x)-u(y)|\;dy \\
&\leq \sum_{i=0}^\infty (2^i r)^{-d-1+\alpha} |B(0,1)| (2^{i+1}r)^d \fint_{B(x,2^{i+1}r)} |u(x)-u(y)|\;dy \\
&\leq \sum_{i=0}^\infty (2^i r)^{-d-1+\alpha} |B(0,1)| (2^{i+1}r)^d \mathcal{M}(u-u(x))(x)
\end{align*}
In particular, we deduce
\begin{align*}
II &\leq C_4 r^{\alpha-1} \mathcal{M}(u-u(x))(x) \\
&\leq 2C_4 r^{\alpha-1} \mathcal{M}(u)(x).
\end{align*}
The result follows from optimizing in $r$, for example with the choice 
\begin{align*}
r= \frac{\mathcal{M}(u)(x)}{\mathcal{M}(|\nabla u|)(x)}.
\end{align*}
\end{proof}

We are now prepared to prove Lemma \ref{isoperimetricinequality}.

\begin{proof}[Proof of Lemma \ref{isoperimetricinequality}]
Let us begin by observing that by Lemma \ref{mazyaslemma} for $u \in C^\infty(\mathbb{R}^d) \cap W^{1,1}(\mathbb{R}^d)$ we have
\begin{align*}
\mathcal{D}^{1-\alpha}(u)(x) &= \int_{\mathbb{R}^d} \frac{|u(x)-u(y)|}{|x-y|^{d+1-\alpha}}\;dy \\
&\leq C \left(\mathcal{M}(|\nabla u|)(x)\right)^{1-\alpha}  \left(\mathcal{M}(u)(x)\right)^{\alpha}.
\end{align*}
Thus we find
\begin{align*}
\|  \mathcal{D}^{1-\alpha}(u) \|_{L^{d/(d-\alpha),1}(\mathbb{R}^d)} &\leq C \| \left(\mathcal{M}(|\nabla u|)(\cdot)\right)^{1-\alpha}  \left(\mathcal{M}(u)(\cdot)\right)^{\alpha} \|_{L^{d/(d-\alpha),1}(\mathbb{R}^d)},
\end{align*}
which in turn by H\"older's inequality in the Lorentz spaces (Theorem \ref{holder} from Section \ref{preliminaries}) we implies
\begin{align*}
\|  \mathcal{D}^{1-\alpha}(u) \|_{L^{d/(d-\alpha),1}(\mathbb{R}^d)} \leq \|C \left(\mathcal{M}(|\nabla u |)\right)^{1-\alpha}\|_{L^{1/(1-\alpha),\infty}(\mathbb{R}^d)} \|\mathcal{M}(u)^\alpha \|_{L^{ d /\alpha(d-1),1}(\mathbb{R}^d)}
\end{align*}
as one checks that
\begin{align*}
\frac{1}{\frac{d}{d-\alpha}} = \frac{d-\alpha}{d} &= 1-\alpha + \alpha - \frac{\alpha}{d} \\
&= \frac{1}{1-\alpha} + \frac{1}{\frac{d}{\alpha(d-1)}}.
\end{align*}
Note here it is crucial that $d>1$.  Next we estimate this from above with the equivalent norm from Section \ref{preliminaries} to observe that
\begin{align*}
\|  \mathcal{D}^{1-\alpha}(u) \|_{L^{d/(d-\alpha),1}(\mathbb{R}^d)}\leq \frac{Cd}{\alpha(d(1-\alpha)+\alpha))} ||| \mathcal{M}(|\nabla u|)^{1-\alpha} |||_{L^{1/(1-\alpha),\infty}(\mathbb{R}^d)} |||\mathcal{M}(u)^\alpha |||_{L^{ d /\alpha(d-1),1}(\mathbb{R}^d)}
\end{align*}
Then the scaling properties of the Lorentz spaces (see Section \ref{preliminaries}), which one has with this equivalent norm, imply 
\begin{align*}
||| \mathcal{M}(|\nabla u|)^{1-\alpha}|||_{L^{1/(1-\alpha),\infty}(\mathbb{R}^d)} &=  ||| \mathcal{M}(|\nabla u|)|||^{1-\alpha}_{L^{1,\infty}(\mathbb{R}^d)} \\
|||\mathcal{M}(u)^\alpha |||_{L^{ d /\alpha(d-1),1}(\mathbb{R}^d)} &= |||\mathcal{M}(u) |||^\alpha_{L^{ d /(d-1),\alpha}(\mathbb{R}^d)}
\end{align*}
Now, the weak-type estimate for the Hardy-Littlewood maximal function recorded in Theorem \ref{weak-type} and the strong-type estimate on the Lorentz space $L^{ d /(d-1),\alpha}(\mathbb{R}^d)$ proven in Theorem \ref{strong-type} (and here note that $\alpha<1$!) implies
\begin{align*}
\| \mathcal{D}^{1-\alpha}(u)\|_{L^{d/(d-\alpha),1}(\mathbb{R}^d)} &\leq C' \left(\int_{\mathbb{R}^d} |\nabla u|\;dx\right)^{1-\alpha}  \|u\|^{\alpha/d}_{L^\infty(\mathbb{R}^d)}  \|u\|_{L^1(\mathbb{R}^d)}^{\alpha(1-1/d)}.
\end{align*}
Now for a set of finite perimeter $E$, define $u_n:= \chi_E \ast \rho_n$ for a sequence of standard mollifiers $\rho_n$.  Then as $u_n \in C^\infty(\mathbb{R}^d) \cap W^{1,1}(\mathbb{R}^d)$, the preceding argument implies
\begin{align*}
\|  \mathcal{D}^{1-\alpha}(u_n)\|_{L^{d/(d-\alpha),1}(\mathbb{R}^d)} &\leq C' \left(\int_{\mathbb{R}^d} |\nabla u_n|\;dx\right)^{1-\alpha}  \|u_n\|^{\alpha/d}_{L^\infty(\mathbb{R}^d)}  \|u_n\|_{L^1(\mathbb{R}^d)}^{\alpha(1-1/d)}.
\end{align*}
We now observe that, up to a subsequence, one has the bound and convergences
\begin{enumerate}
\item[a.]$\|u_n\|_{L^\infty(\mathbb{R}^d)} \leq 1,$
\item[b.]$u_n \to \chi_E \text{ strongly in } L^1(\mathbb{R}^d),$
\item[c.]$\int_{\mathbb{R}^d} |\nabla u_n| \to Per(E),$
\item[d.] $u_n \to u$ pointwise almost everywhere in $\mathbb{R}^d$
\end{enumerate}
and thus Fatou's lemma implies
\begin{align*}
\|  \mathcal{D}^{1-\alpha}(\chi_E) \|_{L^{d/(d-\alpha),1}(\mathbb{R}^d)} &\leq
\liminf_{n \to \infty}\|   \mathcal{D}^{1-\alpha}(u_n)\|_{L^{d/(d-\alpha),1}(\mathbb{R}^d)} \\
&\leq C' \left( Per(E) \right)^{1-\alpha}  |E|^{\alpha(1-1/d)},
\end{align*}
which is the thesis.
\end{proof}

\section{Proofs of the Main Results}\label{proof}
Let us first prove the following theorem, which is the stronger result referred to in the introduction.
\begin{theorem}\label{interpolationinequalityprime}
Let $\alpha \in (0,1)$.  There exists a constant $C=C(\alpha,d)>0$ such that
\begin{align*}
\|\mathcal{D}^{1-\alpha}(u)\|_{L^{d/(d-\alpha),1}(\mathbb{R}^d;\mathbb{R}^d)} \leq C \|D u\|^{1-\alpha}_{L^1(\mathbb{R}^d;\mathbb{R}^d)} \|u\|_{L^{d/(d-1),1}(\mathbb{R}^d)}^\alpha
\end{align*}
for all $u \in BV(\mathbb{R}^d)$.
\end{theorem}

\begin{proof}
We claim that it suffices to prove the inequality for $u \in W^{1,1}(\mathbb{R}^d)$, $u \geq 0$.  To see this, suppose we have proven the inequality for such $u$.  Then utilizing the usual decomposition of a function by its positive and negative parts, $u=u^+-u^-$, we have $\mathcal{D}^{1-\alpha}(u) \leq \mathcal{D}^{1-\alpha}(u^+)+\mathcal{D}^{1-\alpha}(u^-)$. In particular the claimed inequality and the triangle inequality would then imply
\begin{align*}
 \left\| \mathcal{D}^{1-\alpha}(u) \right\|_{L^{d/(d-\alpha),1}(\mathbb{R}^d)} &\leq  \left\| \mathcal{D}^{1-\alpha}(u^+) \right\|_{L^{d/(d-\alpha),1}(\mathbb{R}^d)} +  \left\| \mathcal{D}^{1-\alpha}(u^-) \right\|_{L^{d/(d-\alpha),1}(\mathbb{R}^d)} \\
 &\leq C\|\nabla u^+ \|^{1-\alpha}_{L^1(\mathbb{R}^d;\mathbb{R}^d)}\|u^+\|_{L^{d/(d-\alpha),1}(\mathbb{R}^d)} \\
 &\;\;+  C\|\nabla u^- \|^{1-\alpha}_{L^1(\mathbb{R}^d;\mathbb{R}^d)}\|u^-\|_{L^{d/(d-\alpha),1}(\mathbb{R}^d)}.
\end{align*}
But then one deduces the result for any $u \in W^{1,1}(\mathbb{R}^d)$, up to a slightly larger constant, by the observations
\begin{align*}
\|\nabla u^+ \|_{L^1(\mathbb{R}^d;\mathbb{R}^d)} &\leq \|\nabla u \|_{L^1(\mathbb{R}^d;\mathbb{R}^d)} \\ 
 \|\nabla u^- \|_{L^1(\mathbb{R}^d;\mathbb{R}^d)} &\leq \|\nabla u \|_{L^1(\mathbb{R}^d;\mathbb{R}^d)} \\ 
 \|u^+\|_{L^{d/(d-\alpha),1}(\mathbb{R}^d)} & \leq  \|u\|_{L^{d/(d-\alpha),1}(\mathbb{R}^d)} \\
 \|u^-\|_{L^{d/(d-\alpha),1}(\mathbb{R}^d)} & \leq \|u\|_{L^{d/(d-\alpha),1}(\mathbb{R}^d)} .
\end{align*}
Finally, once we have established the result for $u \in W^{1,1}(\mathbb{R}^d)$, the result for $u \in BV(\mathbb{R}^d)$ follows by density in the strict topology, and using a pointwise convergence and Fatou's lemma to pass the limit for the left-hand-side.

Therefore we restrict our consideration to the case $u \in W^{1,1}(\mathbb{R}^d)$, $u \geq 0$.  Let $E_t$ denote the set $\{ u>t\}$.  Then we can express
\begin{align*}
\mathcal{D}^{1-\alpha}(u) &= \int_{\mathbb{R}^d} \frac{\left| \int_0^\infty \chi_{E_t}(x)- \chi_{E_t}(y) \;dt\right|}{|x-y|^{d+1-\alpha}}\;dy \\
&\leq \int_0^\infty \int_{\mathbb{R}^d} \frac{| \chi_{E_t}(x)- \chi_{E_t}(y)|}{|x-y|^{d+1-\alpha}}\;dydt \\
&= \int_0^\infty \mathcal{D}^{1-\alpha}(\chi_{E_t})(x) \;dt 
\end{align*}
With this equality noted, first an application of Minkowski's inequality for integrals and then an application Lemma \ref{isoperimetricinequality} yields the inequality
\begin{align*}
\left\| \int_0^\infty \mathcal{D}^{1-\alpha}(\chi_{E_t}) \;dt\right\|_{L^{d/(d-\alpha),1}(\mathbb{R}^d)}  &\leq \int_0^\infty \left\| \mathcal{D}^{1-\alpha}(\chi_{E_t}) \right\|_{L^{d/(d-\alpha),1}(\mathbb{R}^d)} \;dt\\
&\leq \int_0^\infty C^\prime Per(E_t)^{1-\alpha} |E_t|^{\alpha(1-1/d)}\;dt.
\end{align*}
But now H\"older's inequality for the integral in $t$ with exponents 
\begin{align*}
\frac{1}{1/(1-\alpha)} + \frac{1}{1/\alpha} = 1
\end{align*}
leads us to conclude
\begin{align*}
 \left\| \mathcal{D}^{1-\alpha}(u) \right\|_{L^{d/(d-\alpha),1}(\mathbb{R}^d)}  \leq C^\prime \left(\int_0^\infty Per(E_t)\;dt\right)^{1-\alpha} \left(\int_0^\infty |E_t|^{1-1/d} \;dt\right)^{\alpha}.
\end{align*}
Finally, by the coarea formula and the definition of the Lorentz space given in Definition 2.3 we have 
\begin{align*}
 \int_0^\infty Per(E_t)\;dt &=  \int_{\mathbb{R}^d} |\nabla u| \\
\int_0^\infty |E_t|^{1-1/d} \;dt &= ||| u |||_{L^{d/(d-1),1}(\mathbb{R}^d)},
\end{align*}
which implies the desired result.
\end{proof}

We next prove Theorem \ref{interpolationinequality}, which follows easily from Theorem \ref{interpolationinequalityprime} and can then be used to deduce Theorem \ref{mainresult}.

\begin{proof} [Proof of Theorem \ref{interpolationinequality}]
Let $u \in BV(\mathbb{R}^d)$ and by a standard approximation argument we may find $\{u_n\} \subset C^\infty_c(\mathbb{R}^d)$ that converges strictly to $u$.  For such $u_n$ we may integrate by parts to obtain
\begin{align*}
|I_\alpha \nabla u_n(x)| &= \frac{1}{\gamma(\alpha)}\left| \int_{\mathbb{R}^d} \frac{\nabla u_n(y)}{|x-y|^{d-\alpha}}\;dy \right|\\
&= \frac{d-\alpha}{\gamma(\alpha)} \left| \int_{\mathbb{R}^d} \frac{u_n(x)-u_n(y)}{|x-y|^{d+1-\alpha}}\frac{x-y}{|x-y|}\;dy\right| \\
&\leq  \frac{d-\alpha}{\gamma(\alpha)} \mathcal{D}^{1-\alpha}(u_n).
\end{align*}
This inequality and Theorem \ref{interpolationinequalityprime} thus imply
\begin{align*}
\| I_\alpha \nabla u_n \|_{L^{d/(d-\alpha),1}(\mathbb{R}^d;\mathbb{R}^d)} \leq C \|\nabla u_n\|^{1-\alpha}_{L^1(\mathbb{R}^d;\mathbb{R}^d)} \|u_n\|_{L^{d/(d-1),1}(\mathbb{R}^d)}^\alpha,
\end{align*}
and since 
\begin{align*}
\|\nabla u_n\|^{1-\alpha}_{L^1(\mathbb{R}^d;\mathbb{R}^d)} &\to \|D u\|^{1-\alpha}_{L^1(\mathbb{R}^d;\mathbb{R}^d)} \\
\|u_n\|_{L^{d/(d-1),1}(\mathbb{R}^d)}^\alpha &\to \|u\|_{L^{d/(d-1),1}(\mathbb{R}^d)}^\alpha,
\end{align*}
as $n \to \infty$, it suffices to show the inequality
\begin{align*}
\| I_\alpha D u \|_{L^{d/(d-\alpha),1}(\mathbb{R}^d;\mathbb{R}^d)} \leq \liminf_{n \to \infty} C\| I_\alpha \nabla u_n \|_{L^{d/(d-\alpha),1}(\mathbb{R}^d;\mathbb{R}^d)}.
\end{align*}
However, for any $j = 1 \ldots d$ and any $\varphi \in C_c(\mathbb{R}^d), \|\varphi\|_{L^{d/\alpha,\infty}(\mathbb{R}^d)}\leq 1$ we have
\begin{align*}
 \left|\int_{\mathbb{R}^d} I_\alpha \frac{\partial u_n}{\partial x_j} \varphi  \right| \leq  \| I_\alpha \nabla u_n \|_{L^{d/(d-\alpha),1}(\mathbb{R}^d;\mathbb{R}^d)}.
\end{align*}
We will manipulate the left-hand-side to a suitable form to pass the limit in this inequality.  First, an application of Fubini's theorem yields the equality
\begin{align*}
  \left|\int_{\mathbb{R}^d} I_\alpha \frac{\partial u_n}{\partial x_j}\; \varphi  \right| = \left|\int_{\mathbb{R}^d}  \frac{\partial u_n}{\partial x_j} I_\alpha\varphi  \right|.
\end{align*}
Next the fact that $\varphi \in C_c(\mathbb{R}^d)$ implies that $I_\alpha \varphi \in C_0(\mathbb{R}^d)$, and so the weak convergence $\nabla u_n \overset{*}{\rightharpoonup} Du$ yields
\begin{align*}
\lim_{n \to \infty} \int_{\mathbb{R}^d}  \frac{\partial u_n}{\partial x_j} \;I_\alpha\varphi  = \int_{\mathbb{R}^d} I_\alpha\varphi \;d(Du)_j.
\end{align*}
Then another application of Fubini's theorem yields
\begin{align*}
 \int_{\mathbb{R}^d} I_\alpha\varphi \;d(Du)_j =  \int_{\mathbb{R}^d} \varphi \;I_\alpha (Du)_j.
\end{align*}
Putting these several steps together we see that for any $j=1 \ldots d$ we have
\begin{align*}
\left| \int_{\mathbb{R}^d} \varphi \;I_\alpha (Du)_j  \right|\leq \liminf_{n \to \infty} \| I_\alpha \nabla u_n \|_{L^{d/(d-\alpha),1}(\mathbb{R}^d;\mathbb{R}^d)}.
\end{align*}
We now utilize the fact that $\varphi \in C_c(\mathbb{R}^d)$ are dense in the weak topology of $L^{d/\alpha,\infty}(\mathbb{R}^d)$ (though not the norm topology!) to recover the norm in $L^{d/(d-\alpha),1}(\mathbb{R}^d)$:
\begin{align*}
\left\| I_\alpha (Du)_j \right\|_{L^{d/(d-\alpha),1}(\mathbb{R}^d)} = \sup_{\varphi \in C_c(\mathbb{R}^d), \|\varphi\|_{L^{d/\alpha,\infty}(\mathbb{R}^d)}\leq 1} \int_{\mathbb{R}^d} I_\alpha (Du)_j\;\varphi.
\end{align*}
Thus we have shown
\begin{align*}
\left\| I_\alpha (Du)_j \right\|_{L^{d/(d-\alpha),1}(\mathbb{R}^d)} \leq C \|\nabla u\|^{1-\alpha}_{L^1(\mathbb{R}^d;\mathbb{R}^d)} \|u\|_{L^{d/(d-1),1}(\mathbb{R}^d)}^\alpha,
\end{align*}
for all $u \in BV(\mathbb{R}^d)$, and the claim follows by summing the components $(Du)_j$ and using the equivalence of norms in finite dimensions.
\end{proof}

We now prove Theorem \ref{mainresult}.

\begin{proof}[Proof of Theorem \ref{mainresult}]
By Lemma 1 of \cite{BonamiPoornima}, the conditions $F \in L^1(\mathbb{R}^d;\mathbb{R}^d)$ and $\operatorname*{curl}F=0$ imply that we may find a sequence $\{u_n\} \subset C^\infty_c(\mathbb{R}^d)$ such that $\nabla u_n \to F$ in $L^1(\mathbb{R}^d;\mathbb{R}^d)$.  The inequality proven in Theorem \ref{interpolationinequality} implies
\begin{align*}
\left\| I_\alpha \nabla u_n \right\|_{L^{d/(d-\alpha),1}(\mathbb{R}^d)} \leq C \|\nabla u_n\|^{1-\alpha}_{L^1(\mathbb{R}^d;\mathbb{R}^d)} \|u_n\|_{L^{d/(d-1),1}(\mathbb{R}^d)}^\alpha,
\end{align*}
which combined with Alvino's Lorentz space inequality \cite{Alvino} yields
\begin{align*}
\left\| I_\alpha \nabla u_n \right\|_{L^{d/(d-\alpha),1}(\mathbb{R}^d)} \leq C \|\nabla u_n\|_{L^1(\mathbb{R}^d;\mathbb{R}^d)}.
\end{align*}
Finally, the convergence $\nabla u_n \to F$ in $L^1(\mathbb{R}^d;\mathbb{R}^d)$ is sufficient to pass the limit on the right-hand-side, while for the left-hand-side we may repeat the argument at the end of Theorem \ref{interpolationinequality} utilizing Fubini's theorem and the weak convergence to conclude the desired result.
\end{proof}

We next prove Theorem \ref{ineq:hardy}.

\begin{proof}
We first prove an analogue of Gagliardo and Nirenberg's inequality between a function and its (fractional) gradient, from which we can easily deduce the desired result.  Thus, let $u$ be such that $D^\alpha u = I_{1-\alpha} \nabla u \in L^1(\mathbb{R}^d;\mathbb{R}^d)$.  Then as $\operatorname*{curl}D^\alpha u=0$, by Theorem \ref{mainresult} we have
\begin{align*}
\|I_\alpha D^\alpha u \|_{L^{d/(d-\alpha),1}(\mathbb{R}^d;\mathbb{R}^d)} \leq C\|D^\alpha u \|_{L^1(\mathbb{R}^d;\mathbb{R}^d)}.
\end{align*}
Now the semi-group property of the Riesz potentials and transforms implies that if $u$ is suitably regular
\begin{align*}
I_\alpha D^\alpha u =  I_1 \nabla u=Ru.
\end{align*}
In particular, in this case the boundedness of $R_j:L^{d/(d-\alpha),1}(\mathbb{R}^d) \to L^{d/(d-\alpha),1}(\mathbb{R}^d)$ implies
\begin{align*}
\|u \|_{L^{d/(d-\alpha),1}(\mathbb{R}^d)} \leq C'\|D^\alpha u \|_{L^1(\mathbb{R}^d;\mathbb{R}^d)},
\end{align*}
which is the desired inequality for sufficiently regular functions.  The case of general functions follows easily here by again invoking Bonami and Poornima's approximation argument \cite{BonamiPoornima}.  Finally, the claimed Hardy inequality follows easily from H\"older's inequality in the Lorentz spaces, as
\begin{align*}
 \int_{\mathbb{R}^d} \frac{|u|\;}{|x|^{\alpha}}\;dx \leq \|u\|_{L^{d/(d-\alpha),1}(\mathbb{R}^d)} \left\| \frac{1}{|\cdot|^\alpha}\right\|_{L^{d/\alpha,\infty}(\mathbb{R}^d)},
 \end{align*}
 and 
 \begin{align*}
 \left\| \frac{1}{|\cdot|^\alpha}\right\|_{L^{d/\alpha,\infty}(\mathbb{R}^d)} \leq C''.
\end{align*}
\end{proof}

Finally, we conclude with a proof of the dual result claimed in the introduction.

\begin{proof}[Proof of Corollary \ref{dual}]
Define the space
\begin{align*}
X := \left\{ f \in \mathcal{D}^\prime(\mathbb{R}^d) : Rf \in L^1(\mathbb{R}^d;\mathbb{R}^d)\right\},
\end{align*}
which we equip with the norm
\begin{align*}
\|f\|_X := \|Rf \|_{L^1(\mathbb{R}^d;\mathbb{R}^d)}.
\end{align*}
Then we can identify the topological dual of $X$, $X^\prime$, with
\begin{align*}
X^\prime = \left\{ g \in \mathcal{D}^\prime(\mathbb{R}^d) : g = \sum_{j=1}^d R_j Y_j \text{ for some } \{Y_j\}_{j=1}^d \subset L^\infty(\mathbb{R}^d)\right\},
\end{align*}
where
\begin{align*}
\|g\|_{X^\prime} = \inf \left\{ \||Y|\|_{L^\infty(\mathbb{R}^d)}  : g = \sum_{j=1}^d R_j Y_j \text{ for some } \{Y_j\}_{j=1}^d \subset L^\infty(\mathbb{R}^d)\right\}. 
\end{align*}

Thus it suffices to show the estimate
\begin{align*}
\|I_\alpha g \|_{X^\prime} \leq C \|g\|_{L^{d/\alpha,\infty}(\mathbb{R}^d)}.
\end{align*}
However this follows directly by the standard duality argument.  In particular, we have
\begin{align*}
\| I_\alpha g \|_{X^\prime} =  \sup_{f} \int_{\mathbb{R}^d} I_\alpha g f\;dx
\end{align*}
where the supremum is taken over all functions $f \in X,  \|f\|_{X} \leq 1$.  However, now the fact that the Riesz potential is (up to a minus sign) self-adjoint and the introduction of the Riesz transforms $R$ yields the equality
\begin{align*}
\int_{\mathbb{R}^d} I_\alpha g f\;dx = -\int_{\mathbb{R}^d} Rg \cdot I_\alpha Rf\;dx.
\end{align*}
But $\operatorname*{curl} Rf=0$, and thus Theorem \ref{mainresult}, along with the boundedness of the Riesz transforms on $L^{d/\alpha,\infty}(\mathbb{R}^d)$ yields the inequality
\begin{align*}
\left|\int_{\mathbb{R}^d} Rg \cdot I_\alpha Rf\;dx \right| &\leq \|I_\alpha Rf\|_{L^{d/(d-\alpha),1}(\mathbb{R}^d)} \|Rg\|_{L^{d/\alpha,\infty}(\mathbb{R}^d;\mathbb{R}^d)}  \\
&\leq C \|Rf\|_{L^1(\mathbb{R}^d;\mathbb{R}^d)}  \|g\|_{L^{d/\alpha,\infty}(\mathbb{R}^d)}\\
&= C \|f\|_{X} \|g\|_{L^{d/\alpha,\infty}(\mathbb{R}^d)},
\end{align*}
which shows that for $g \in L^{d/\alpha,\infty}(\mathbb{R}^d)$, $I_\alpha g \in X^\prime$ with the desired norm bound.
\end{proof}


\section*{Acknowledgements}
The author would like to thank Vladimir Maz'ya, Chun-Yen Shen, and Shiah-Sen Wang for the stimulating conversations during the undertaking of the this research, Armin Schikorra and Jean Van Schaftingen for their reading of and comments on preliminary versions of this manuscript, and Aline Bonami and Mario Milman for discussions regarding the optimal Lorentz estimate for functions in the Hardy space $\mathcal{H}^1(\mathbb{R}^d)$.  Needless to say that I remain responsible for the remaining shortcomings.  The author is supported in part by the Taiwan Ministry of Science and Technology under research grants 105-2115-M-009-004-MY2, 107-2918-I-009-003 and 107-2115-M-009-002-MY2.


\end{document}